\documentclass[reqno]{amsart} 
\usepackage{amscd,amssymb,latexsym,subfigure,verbatim,hyperref,epsfig,amsmath, xypic,supertabular}
\usepackage{tikz}
\usetikzlibrary{matrix, arrows}
\usepackage{tikz-cd}
\usepackage{enumitem}
\DeclareGraphicsExtensions{.pdf,.jpg, .png}
\usepackage{psfrag}
\usepackage{multicol}

\xyoption{all}
\usepackage{graphicx,amscd,amssymb,latexsym,subfigure,verbatim,hyperref,epsfig,supertabular}
\usepackage[mathscr]{eucal}
\usepackage{mathrsfs}
\usepackage{fullpage}
\usepackage{graphicx}
\usepackage{hyperref}
\usepackage{amssymb}

\newcommand{\K}{{\mathbb K}}

\newcommand{\p}{{\mathbb P}}

\newcommand{\OX}{{\mathcal O}_X}

\newcommand{\depth}{\mathop{\rm depth}\nolimits}

\newcommand{\codim}{\mathop{\rm codim}\nolimits}
\newcommand{\Pfaff}{\mathop{\rm Pfaff}\nolimits}

\newcommand{\rank}{\mathop{\rm rank}\nolimits}
\newcommand{\reg}{\mathop{\rm regularity}\nolimits}

\newcommand{\coker}{\mathop{\rm coker}\nolimits}
\newtheorem{defn0}{Definition}[section]
\newtheorem{prop0}[defn0]{Proposition}
\newtheorem{conj0}[defn0]{Conjecture}
\newtheorem{thm0}[defn0]{Theorem}
\newtheorem{lem0}[defn0]{Lemma}
\newtheorem{corollary0}[defn0]{Corollary}
\newtheorem{example0}[defn0]{Example}
\newtheorem{question0}[defn0]{Question}
\newtheorem{remark0}[defn0]{Remark}
\newtheorem{computation0}[defn0]{Computation}

\newenvironment{defn}{\begin{defn0}}{\end{defn0}}

\newenvironment{prop}{\begin{prop0}}{\end{prop0}}

\newenvironment{thm}{\begin{thm0}}{\end{thm0}}
\newenvironment{lem}{\begin{lem0}}{\end{lem0}}

\newenvironment{exm}{\begin{example0}\rm}{\end{example0}}
\newenvironment{rmk}{\begin{remark0}\rm}{\end{remark0}}

\numberwithin{equation}{section}
\title[Calabi-Yau threefolds in $\p^n$ and Gorenstein rings]%
{Calabi-Yau threefolds in $\p^n$ and Gorenstein rings}
\author{Hal Schenck}
\thanks{Schenck supported by NSF 1818646}
\address{Schenck: Mathematics
	Department \\ Auburn University\\
	Auburn \\ AL 36849\\ USA}
\email{hks0015@auburn.edu}

\author{Mike Stillman}
\thanks{Stillman and Yuan supported by NSF 1502294}
\address{Stillman: Mathematics Department \\ Cornell University \\
	Ithaca \\ NY 14850\\ USA}
\email{mike@math.cornell.edu}

\author{Beihui Yuan}
\address{Yuan: Mathematics Department \\ Cornell University \\
	Ithaca \\ NY 14850\\ USA}
\email{by238@math.cornell.edu} 


\begin{document}
	\begin{abstract}
		\noindent 
		A projectively normal Calabi-Yau threefold $X \subseteq \p^n$ has an ideal $I_X$ which is arithmetically
		Gorenstein, of Castelnuovo-Mumford regularity four. Such ideals have
		been intensively studied when $I_X$ is a complete intersection, as
		well as in the case where $X$ is codimension three. In the latter
		case, the Buchsbaum-Eisenbud theorem shows that $I_X$ is given by the
		Pfaffians of a skew-symmetric matrix. A number of recent papers study the situation when $I_X$ has codimension
		four. We prove there are 16 possible betti tables for an arithmetically Gorenstein ideal $I$ with $\codim(I)=4=\reg(I)$, and that exactly 8 of these occur for smooth irreducible nondegenerate threefolds. We investigate the situation in codimension five or more, obtaining 
		examples of $X$ with $h^{p,q}(X)$ not among those appearing for $I_X$ of lower codimension or as complete intersections in toric Fano varieties. A key tool in our approach is the use of inverse systems to identify possible betti tables for $X$.
		
	\end{abstract}
	\maketitle
	
	\section{Introduction}\label{sec:intro}
	In their 1985 paper \cite{chsw}, Candelas-Horowitz-Strominger-Witten showed that Calabi-Yau threefolds play a central role in string theory.
	This was further developed in works by Candelas-Lynker-Schimmrigk \cite{cls} and  Candelas-de la Ossa-Green-Parkes \cite{cogp}; the book of Cox-Katz \cite{CK} gives a comprehensive overview of the field.
	\begin{defn}
		A smooth variety $X$ of dimension $n$ is Calabi-Yau if $K_X \simeq \OX$ and
		\[
		H^1(\OX), \cdots, H^{n-1}(\OX) = 0.
		\]
	\end{defn}
	From the perspective of physics, the case $n=3$ is of paramount interest, and a first example of a CY threefold is a quintic hypersurface in $\p^4$.
	Generalizing the hypersurface case, when $X$ is a complete intersection (CI) of type $\{d_1, \ldots, d_{n-3}\} \subseteq \p^n$ we have
	\[
	K_X \simeq \OX(-n-1+\sum d_i).
	\]
	So a complete intersection Calabi-Yau  (CICY) threefold in $\p^n$ must have $\{d_1, \ldots, d_{n-3}\}$ satisfying
	\[\begin{array}{ccc}
		\{5\} & \mbox{ in }& \p^4\\
		\{2,4\} & \mbox{ in }& \p^5\\
		\{3,3\} & \mbox{ in }& \p^5\\
		\{2,2,3\} & \mbox{ in }& \p^6\\
		\{2,2,2,2\} & \mbox{ in }& \p^7
	\end{array}
	\]
	Green-H\"ubsch-L\"utken characterize CICY's $X \subset \prod_{i=1}^m \p^{n_i}$ in \cite{GHL}; when $m=1$, $h^{1,1}(X)=1$ and $h^{1,2}(X) \in \{65,73,89,101\}$. 
	Projective space is the simplest complete toric variety \cite{cox}, and in \cite{B}, Batyrev shows how to obtain
	CY's as hypersurfaces in toric varieties  corresponding to reflexive
	polytopes. Much activity over the last three decades has been devoted to this situation. A complete intersection
	is the first avatar of a Gorenstein ring; a Gorenstein ideal of codimension two is a complete intersection, and Buchsbaum-Eisenbud
	\cite{BE} show that a codimension three Gorenstein ideal is generated
	by the Pfaffians of a skew-symmetric matrix. From the CY perspective,
	this is investigated in \cite{r}, \cite{T} and subsequent papers. The
	codimension four case was first studied systematically by Bertin in \cite{Ber}; in \cite{cgkk} 
	Coughlan-Golebiowski-Kapustka-Kapustka list 11 Gorenstein Calabi-Yau (GoCY) threefolds in $\p^7$
	and ask if the list is complete. For other recent work on GoCY's, see \cite{BG}, \cite{BKZ}, and \cite{reid}. 

	\subsection{Preliminaries}
	For algebraic background, we refer to \cite{E}. 
	The first observation to make is that if $S=\K[x_0,\ldots x_n]$ and $I$ is a nondegenerate (i.e. containing no linear form)  homogeneous ideal in $S$ such that $R=S/I$ is arithmetically Cohen-Macaulay, then the canonical module $\omega_R$ is isomorphic to a shift $R(a)$ exactly when $R$ is arithmetically Gorenstein (henceforth {\em Gorenstein}).
	In general, 
	\[
	\omega_R \simeq R(-n-1+\reg(R)+\codim(R)), \mbox{ so we have}
	\]
	\begin{lem}
		If $X = Proj(R)$ is a arithmetically Cohen-Macaulay threefold, then
		\begin{equation}\label{Eq1}
			K_X = \OX \longleftrightarrow -n-1 +n-3 +\reg(R) = 0 \longleftrightarrow \reg(R)=4.
		\end{equation}
	\end{lem}
	For $R$ Gorenstein, we may quotient by a regular sequence of linear forms, reducing to an Artinian Gorenstein ring with the same
	homological behavior, described below. Any Artinian Gorenstein ring arises (\cite{E},
	\S 21) via  Macaulay's {\em inverse system} construction: Let $F$ be a homogeneous
	polynomial in $S$ of degree $d$ over a field $\K$ of characteristic
	zero. The set of differential operators $P(\frac{\partial}{\partial x_0},\ldots,
	\frac{\partial}{\partial x_n})$ which annihilate $F$ generates an
	ideal $I_F$ in $T = \K[\frac{\partial}{\partial
		x_0},\ldots,\frac{\partial}{\partial x_n}]$, which is called the
	inverse system $I_F \subseteq T$. The corresponding ring $T/I_F$ is an
	Artinian Gorenstein ring of regularity $d$.
	\begin{defn}\label{bettiTable}
		For $I \subseteq S= \K[x_0,\ldots, x_n]$ homogeneous, the graded betti numbers are
		\[
		b_{ij} = \dim_{\K}Tor_i(S/I,\K)_j.
		\]
		In betti table notation \cite{e3}, these numbers are displayed as an array with top left entry in position $(0,0)$ and position $(i,j)$ equal to $b_{i, i+j}$. The reason for this indexing is so that 
		\[
		\reg(S/I) = \sup\{j \mid b_{i,i+j} \ne 0\}
		\]
		is given by the index of the bottom row of the betti table. 
	\end{defn}
	\begin{exm}
		For a GoCY threefold $X \subseteq \p^6$ given by the Pfaffians of a skew-symmetric $7 \times 7$ matrix $M$ of generic linear forms, 
		R\o dland \cite{r} shows that $h^{1,1}(X)=1$ and $h^{1,2}(X)=50$. 
		By \cite{BE} and \cite{em}, for a generic quartic $F$ in $\K[x_0,x_1,x_2]$, $I_F$ is the Pfaffians of a skew $7 \times 7$ matrix $M$ of linear forms, with betti table below. So for this example, the betti table of $I_F$ can be realized by a GoCY. 
		\begin{small}
			$$
			\vbox{\offinterlineskip 
				\halign{\strut\hfil# \ \vrule\quad&# \ &# \ &# \ &# \ &# \ &# \
					&# \ &# \ &# \ &# \ &# \ &# \ &# \
					\cr 
					total&1&7&7&1\cr 
					\noalign {\hrule}
					0&1 &--&--& --&    \cr 
					1&--&-- &-- & --&        \cr 
					2&--&7 &7 &--&        \cr 
					3&--&-- &-- &--&         \cr 
					4&--&-- &-- &1&         \cr 
					\noalign{\bigskip}
					\noalign{\smallskip}
			}}
			$$
		\end{small}
		\noindent In \cite{T}, Tonoli finds smooth CY's $X \subseteq \p^6$ with $12 \le \deg(X) \le 17$.  
		By Lemma~\ref{DegX} below, only those with $\deg(X)\le 14$ are GoCY's; all have $h^{1,1}(X)=1$. 
		Ordered by increasing degree of $X$ from $\{12,\ldots ,17\}$, the respective $h^{1,2}(X)$ are $\{73,61,50,40,31,23 \}$. 
	\end{exm} 
	
	\noindent For $S/I$ Artin Gorenstein of regularity 4, the Hilbert function
	of $S/I$ is \[
	HF(S/I,t) = \Big(1, n+1, h_2, n+1, 1\Big), \mbox{ with } h_2 \le {n+2 \choose 2}. 
	\]
	Migliore-Zanello \cite{MZ} show that Stanley's example of a Gorenstein ring with non-unimodal $H$-vector $(1,13,12,13,1)$ is minimal, so for $n \le 11$, $n+1 \le h_2$. 
	If $I$ can be lifted to a prime ideal in four more variables (Example~\ref{Ex1} shows this can occur), the corresponding threefold $X$ will have degree  
	\begin{equation}\label{Eq2}
		\deg(X) = \sum_t HF(S/I,t) = 2n+4+h_2. 
	\end{equation}
	\begin{lem}\label{DegX}A GoCY $ X \subseteq  \p^n$ with $n \le 15$ has 
		\begin{equation}\label{Eq3}
			3n-7 \le \deg(X) \le \frac{n^2-n-2}{2}.
		\end{equation}
	\end{lem}
	\begin{proof}
		Apply Equations~\ref{Eq1} and \ref{Eq2} and the result of \cite{MZ} to the Artinian reduction of $S/I_X$.
	\end{proof}
	\noindent This generalizes Lemma 2.1 of \cite{cgkk}, which shows that a GoCY $X \subseteq \p^7$ has $14 \le \deg(X) \le 20$. 
	
	\section{Gorenstein codimension four}
	There have been several recent works on GoCY threefolds in $\p^7$. In \cite{Ber}, Bertin finds four distinct families. In addition to Reid's
	work \cite{reid} on a structure theory for Gorenstein codimension four ideals, work on GoCY's in $\p^7$ appears in \cite{BG}, \cite{BKZ}, and \cite{cgkk}, where Coughlan-Golebiowski-Kapustka--Kapustka give a table with 11 types of GoCY threefolds in $\p^7$. They ask if the list is complete; on their list $h^{1,1}(X) \in \{1,2\}$ and $h^{1,2}(X) \in \{34,36,37,45,46,54,55,58,65,76,86\}$. In this section, we use Betti tables to explore their question. The $h^{p,q}(X)$ cannot be read from the betti table of $S/I_X$: the two families of degree 18 GoCY's appearing in \cite{cgkk} have the same betti table, but $h^{1,2}(X)$ is either $45$ or $46$. 
	
	\begin{exm}\label{Ex1}The inverse system of a generic quartic in four variables yields an ideal with betti table labelled CGKK 11 below. 
		For an $n \times n$ matrix $M$ with $M_{i,j} = x_{ij}$, Gulliksen-Neg\.ard 
		\cite{GN} determine the resolution of the ideal $I_{n-1}$ of $n-1 \times n-1$ minors: it is Gorenstein of codimension 
		four, and has regularity $2n-4$. Hence if $n=4$, this yields a Gorenstein codimension 
		four ideal in $\p^{15}$. Quotienting with a regular sequence of eight 
		linear forms yields a smooth GoCY threefold in $\p^7$, with betti diagram 
		equal to that of CGKK 11. The Hodge numbers are $h^{1,1}=2$ and $h^{1,2}=34$; this example was first identified by Bertin in \cite{B}. 
	\end{exm}

	\begin{thm}\label{possibleBetti}
		An Artin Gorenstein algebra $A=S/I$ with $\reg(A) = 4 = \codim(A)$ and $I$ nondegenerate has one of the 16 betti diagrams below. Table 1 below corresponds to the 11 classes of GoCY in \cite{cgkk}, and Table 2 to the remaining classes. We defer the proof until the end of this section.

	\end{thm}
	\begin{table}[ht]
		\begin{verbatim}
			CGKK 1             1 9 16 9 1              CGKK 5,6      1 9 16 9 1
			                0: 1 .  . . .                         0: 1 . .  . .
			                1: . 6  8 3 .                         1: . 3 2  . .
			                2: . .  . . .                         2: . 6 12 6 .
			                3: . 3  8 6 .                         3: . . 2  3 .
			                4: . .  . . 1                         4: . . .  . 1
			
			CGKK 2             1 6 10 6 1              CGKK 7,8      1 10 18 10 1 
			                0: 1 .  . . .                         0: 1  .  .  . .
			                1: . 5  5 . .                         1: .  2  .  . .
			                2: . 1  . 1 .                         2: .  8 18  8 .
			                3: . .  5 5 .                         3: .  .  .  2 .
			                4: . .  . . 1                         4: .  .  .  . 1
			
			CGKK 3             1 4 6 4 1               CGKK 9,10     1 13 24 13 1
			                0: 1 . . . .                          0: 1  .  .  . .
			                1: . 4 . . .                          1: .  1  .  . .
			                2: . . 6 . .                          2: . 12 24 12 .
			                3: . . . 4 .                          3: .  .  .  1 .
			                4: . . . . 1                          4: .  .  .  . 1
			
			CGKK 4             1 7 12 7 1              CGKK 11       1 16 30 16 1 
			                0: 1 .  . . .                         0: 1  .  .  . .
			                1: . 3  . . .                         1: .  .  .  . .
			                2: . 4 12 4 .                         2: . 16 30 16 .
			                3: . .  . 3 .                         3: .  .  .  . .
			                4: . .  . . 1                         4: .  .  .  . 1
			
		\end{verbatim}
		\caption{\textsf{Betti table for GoCY's in Coughlan-Golebiowski-Kapustka–Kapustka}}
	\end{table}
	\vskip -.3in
	\noindent The first four columns on the left have degrees $\{14,15,16,17\}$, while the columns on the right have degrees $\{17,18,19,20\}$, as noted above the corresponding GoCY's can have different Hodge numbers. 
	
	\pagebreak
	\noindent The other 8 betti tables for an Artin Gorenstein algebra $A$ with $\reg(A) = 4 = \codim(A)$ are:
	\begin{table}[ht]
		\begin{verbatim}
			Type 2.1           1 11 20 11 1            Type 2.5      1 11 20 11 1
			                0: 1  . .  .  .                       0: 1 .  .  .  .
			                1: .  2 1  .  .                       1: . 3  3  1  .
			                2: .  9 18 9  .                       2: . 7  14 7  .
			                3: .  . 1  2  .                       3: . 1  3  3  .
			                4: .  . .  .  1                       4: . .  .  .  1
			
			Type 2.2           1 8 14 8 1              Type 2.6      1 9 16 9 1 
			                0: 1 . .  . .                         0: 1 . .  . .
			                1: . 3 1  . .                         1: . 4 4  1 .
			                2: . 5 12 5 .                         2: . 4 8  4 .
			                3: . . 1  3 .                         3: . 1 4  4 .
			                4: . . .  . 1                         4: . . .  . 1
			
			Type 2.3           1 7 12 7 1              Type 2.7      1 7 12 7 1
			                0: 1 . . . .                          0: 1 .  . . .
			                1: . 4 3 . .                          1: . 5  5 1 .
			                2: . 3 6 3 .                          2: . 1  2 1 .
			                3: . . 3 4 .                          3: . 1  5 5 .
			                4: . . . . 1                          4: . .  . . 1
			
			Type 2.4           1 6 10 6 1              Type 2.8      1 9 16 9 1 
			                0: 1 . .  . .                         0: 1 . .  . .
			                1: . 4 2  . .                         1: . 5 6  2 .
			                2: . 2 6  2 .                         2: . 2 4  2 .
			                3: . . 2  4 .                         3: . 2 6  5 .
			                4: . . .  . 1                         4: . . .  . 1
			
		\end{verbatim}
		\caption{\textsf{Betti tables for the remaining 8 Artin Gorenstein algebras.}}
	\end{table}
	\vskip -.5in
	\begin{thm}
		No diagram in Table 2 occurs for a smooth, irreducible nondegenerate threefold in $\p^7$.
	\end{thm}
	\begin{proof}
		In \S 3.4 we apply results of \cite{VV} to prove a structure theorem for any irreducible nondegenerate threefold in $\p^7$ with betti diagram of Type 2.4, and show the resulting variety cannot be smooth. For the other betti tables, we apply a result of \cite{ss}.  A matrix of linear forms is {\em 1-generic} if no entry can be reduced to zero by (scalar) row or column operations; a linear $n^{th}$ syzygy is an element of $Tor^S_{n+1}(S/I,\K)_{n+2}$. For a nondegenerate prime ideal $P$, Theorem 1.7 of \cite{ss} shows:
		\begin{enumerate}
			\item $P$ cannot have a linear $n^{th}$ syzygy of rank $\le n+1$, or $P$ is not prime.
			\item If $P$ has a linear $n^{th}$ syzygy of rank $n+2$, then $P$ contains the $2 \times 2$ minors of a 1-generic $2 \times (n+2)$ matrix.
			\item If $P$ has a linear $n^{th}$ syzygy of rank $n+3$, then $P$ contains the $4 \times 4$ Pfaffians of a skew-symmetric 1-generic $n+4 \times n+4$ matrix.
		\end{enumerate}
		A betti table of Type 2.1 is ruled out by $(1)$, and a betti table of Type 2.2 is ruled out by $(2)$, since the $2 \times 2$ minors of a $2 \times 3$ matrix have two independent linear syzygies. For the three betti tables having top row of the form $(c, c, 1)$, we argue as follows. When $c=3$, the linear second syzygy can have rank at most 3, since it involves the 3 first syzygies. Hence by $(1)$, the ideal cannot be prime. When $c=4$,  the linear second syzygy can have rank at most 4, and in this case by $(2)$ it contains the $2\times 2$ minors of a 1-generic $2\times 4$ matrix, which would yield a top row of the betti table with entries $(6, 8, 3)$. When $c=5$, $(3)$ implies that $P$ contains the Pfaffians, and since there are only five quadrics, the quadratic part of the idea is exactly the Pfaffians, which do not have a linear second syzygy. 
		
		For Type 2.3, we will show that a prime non-degenerate ideal $P$ cannot have top row of the betti table equal to $(4, 3, 0)$. Let $J_2$ be the subideal of $P$ generated by quadrics in $P$.
		By $(1)$ and $(3)$ the first syzygies all have rank three; take a subideal $I\subseteq J_2$ consisting of three elements, which by $(2)$ is generated by the $2 \times 2$ minors of a $2 \times 3$ matrix, and let $F$ denote the remaining quadric, so $J_2=I+\langle F\rangle$. Consider the mapping cone resolution of $S/J_2$ from the short exact sequence
		\begin{equation}\label{sesQuadric}
			0\longrightarrow S(-2)/I:F \longrightarrow S/I\longrightarrow S/I+F \longrightarrow 0.
		\end{equation}
		It follows that $I:F$ must have a linear generator $L$, so $LF \in I$. If $I$ is prime, then either $L \in I$ or $F \in I$, a contradiction.
		So suppose $I$ is not prime, and take a primary decomposition 
		\[
		I = \cap_{i=1}^m Q_i, \mbox{ with }\sqrt{Q_i}=P_i \mbox{ all  codimension two.}
		\]
		Since $I$ is codimension two and Cohen-Macaulay and $\deg(I) = 3$, we must have $m \le 3$.
		\begin{enumerate}
			\item Case 1: $m=3$. Then $Q_i=P_i$ and $I = \cap_{i=1}^3 P_i$ with $P_i$ generated by two linear forms.
			\item Case 2: $m=2$. Then $\deg(Q_1)=1$, $\deg(Q_2)=2$, so $\sqrt{Q_1}$ is generated by two linear forms. 
			\item Case 3: $m=1$. Then $\sqrt{Q_1}=P_1$, with $\deg(P_1) \in \{1,2,3\}$. If $\deg(P_1)=3$, then $I$ is prime, and if $\deg(P_1) = 1$ or $2$, $P_1$ contains a linear form. 
		\end{enumerate}
		In particular, we see that $P$ is degenerate. For Type 2.8, both second syzygies must have rank six, because if either had lower rank, then we would be in one of the cases $(1),(2),(3)$, all of which are inconsistent with a betti table having top row $(5,6,2)$. Let $M$ denote the corresponding $6 \times 2$ matrix of linear second syzygies; $M$ is 1-generic: if not, there is a second syzygy of rank $\le 5$, a contradiction. We claim that $M^t$ has no linear first syzygies (notice that kernel of $M^t$ will contain any linear first syzygies on $J_2$, which is the subideal of $P$ generated by quadrics in $P$). 
		This follows because $\coker(M^t)$ has a Buchsbaum-Rim resolution \cite{E}, Theorem A2.10. The Buchsbaum-Rim complex is a 
		resolution for $\coker(M^t)$ iff the $2 \times 2$ minors of $M$ have depth $6-2+1 = 5$; since $M$ is 1-generic, the $2 \times 2$ minors are Cohen-Macaulay with an Eagon-Northcott resolution; in particular $\depth(I_2(M))=5$. As the first syzygies in the Buchsbaum-Rim complex for $M^t$ come from $\Lambda^3(S^6)$ via the splice map described in \cite{E}, they are quadratic. We conclude there are no linear syzygies on $\coker(M^t)$, hence no linear first syzygies on $J_2$, a contradiction.
	\end{proof}
	\noindent For the proof of Theorem 2.2, we will need the theorems of Macaulay and Gotzmann \cite{sch}: for a graded algebra $S/I$ with  Hilbert function $h_i$, write 
	\[
	h_i = \binom{a_i}{i} + \binom{a_{i-1}}{i-1}+ \cdots \mbox{ and }h_i^{\langle i \rangle} = \binom{a_i+1}{i+1} + \binom{a_{i-1}+1}{i}+ \cdots, \mbox{ with } a_i > a_{i-1} > \cdots
	\]
	Macaulay proved that $h_{i+1} \le h_i^{\langle i \rangle}$, and Gotzmann proved if $I$ is generated in a single degree $t$ and equality holds in Macaulay's formula in the first degree $t$, then 
	\[h_{t+j} = \binom{a_t+j}{t+j} + \binom{a_{t-1}+j-1}{t+j-1}+ \cdots\]
	We also need the following lemma
	\begin{lem}\label{lem_v_neq_(2,1)}
		Let $I_2$ be the subideal of $I$ generated by the quadrics in $I$, and let $v=(b_{23},b_{24})$, which are the number of linear first and second syzygies on $I_2$. Then 
		\begin{enumerate}
			\item[(a)] $v\neq (2,1)$. 
			\item[(b)] if $a=b_{12}\geq 4$, then $v\neq (3,1)$.
		\end{enumerate}
	\end{lem}
	\begin{proof}
		To see that $v=(2,1)$ cannot occur, observe that if it did then there would be a unique relation $L_1 \cdot  V_1 + L_2 \cdot V_2 = 0$ where $L_1, L_2$ are linear forms, and $V_i$ are vectors of linear first syzygies. Changing variables so $L_1 = x_1 \mbox{ and }L_2=x_2$, we have that $x_1 \cdot V_{i1} + x_2 \cdot V_{i2} = 0$ for all $i$, implying $V_1$ is $x_2 \cdot C$ and $ V_2$ is $-x_1 \cdot C$, with $C$ a vector of constants, a contradiction. So $v=(2,1)$ is impossible.\\
		To prove part (b), the key point is that $v=(3,1)$ implies that $I_2$ contains 
		$\{Lx_1,Lx_2,Lx_3\}$ with $L$ a linear form. When $a \ge 4$ the mapping cone construction implies $I_2$ is inconsistent with the Gorenstein hypothesis (IGH). If $v=(3,1)$ then the unique linear second syzygy $S$ must have rank 3, otherwise the argument showing that $v=(2,1)$ is impossible applies. After change of variables, we may write $S$ as below, with $a_i,b_i,c_i$ linear forms:
		\begin {center}
		$\left[ \!
		\begin{array}{ccc}
			a_1 & b_1  &c_1   \\
			a_2 & b_2  &c_2   \\
			a_3 & b_3  &c_3   \\
			a_4 & b_4  &c_4
		\end{array}\! \right] \cdot 
		\left[ \!
		\begin{array}{c}
			x_1\\
			x_2\\
			x_3
		\end{array}\! \right] = 0$
	\end{center}
	So the rows of the matrix of linear first syzygies on $I_2$ are Koszul syzygies on $[x_1,x_2,x_3]^t$, that is to say
	\begin{align*}
		\begin{bmatrix}a_1 & b_1  &c_1   \\
			a_2 & b_2  &c_2   \\
			a_3 & b_3  &c_3   \\
			a_4 & b_4  &c_4\end{bmatrix}=C\begin{bmatrix}
			x_2 & -x_1 &0\\
			-x_3 & 0 & x_1\\
			0 & x_3 & -x_2
		\end{bmatrix}
	\end{align*}
	where $C$ is a full rank $4\times 3$ scalar matrix.
	This forces 
	$I_2$ to contain $\{Lx_1,Lx_2,Lx_3\}$. \\
	If $a \ge 4$, $I_2$ must contain a quadric $Q$ which is a nonzero divisor on $\{Lx_1,Lx_2,Lx_3\}$.
	To see this, note that if $Q \in \langle L\rangle$ then $\codim(I_2) =1$. After a change of variables $I_2$ consists of a linear form times a subset of the variables, so that $I_2$ has a Koszul resolution, hence $b_{45}(I_2) \ne 0$ which is IGH; if $Q \in \langle x_1,x_2,x_3 \rangle$ then there is at least one additional linear first syzygy, so $b \ge 4$.
	Now we know $Q$ must be a non-zero divisor on $\{Lx_1,Lx_2,Lx_3\}$. This implies that if $v=(3,1)$, then $I_2$ has mapping cone betti table
	\begin {center}
	$ \left[ \!
	\begin{array}{ccccc}
		1& 0 &0 &0 &0 \\
		0&  4 & 3  &1  &0 \\
		0&  0 & 3  &3  &1 
	\end{array}\! \right]. $
\end{center}
This is IGH, because $Tor_4(R/I_2,\K)_6 \ne 0$, and adding additional generators to $I_2$ cannot force cancellation: for a cubic $F$, we have the short exact sequence
\begin{equation}\label{sesCubic}
	0\longrightarrow R(-3)/I_2:F \longrightarrow R/I_2\longrightarrow R/I_2+F \longrightarrow 0
\end{equation}
and the associated long exact sequence gives exact sequence of vector spaces
\begin{align*}
	0\to Tor_{4}(R(-3)/I_2:F,\K)_{6}\to Tor_{4}(R/I_2,\K)_{6}\to Tor_{4}(R/I_2+F,\K)_{6}.
\end{align*}
Note that
\begin{align*}
	Tor_4 (R(-3)/I_2:F,\K)_6 = Tor_4 (R/I_2:F,\K)_3 = 0.
\end{align*}
Hence $Tor_{4}(R/I_2,\K)_{6}\ne 0$ implies $Tor_{4}(R/I_2+F,\K)_{6}\ne 0$.
Therefore we conclude $v=(3,1)$ is IGH.
\end{proof}
\begin{rmk}
When $a=3$, $v=(3,1)$ occurs.
\end{rmk}
Now we can prove Theorem 2.2:
\begin{proof}[Proof of Theorem 2.2:] We use the Hilbert function to establish the possible shape of the betti table, combined with an analysis of the structure of the subideal $I_2$ generated by the quadrics in $I$ and subideal $C_3$ generated by the quadrics and cubics in $I$. Let $a=b_{12}(I)$ be the number of quadratic generators of $I\subseteq R = \K[x_1,\ldots,x_4]$, and let $v=(b_{23}, b_{34})=(b,c)$, which are the number of linear first and second syzygies on $I_2$.  Note that $b_{45}(I_2) \ne 0$ is inconsistent with the Gorenstein hypothesis (IGH), so cannot occur. 

For an Artinian Gorenstein ideal $I$ with $\codim(I) = 4 = \reg(4)$ and fixed Hilbert series, $v$ determines the entire betti table. If $\codim(I_2)=1$ then after a change of variables $I_2$ consists of a linear form times a subset of the variables, so that $I_2$ has a Koszul resolution; in particular for $a \ge 4$ this is IGH. Similarly, $\codim(I_2)=4$ can only occur if $I_2$ contains a complete intersection. When $a \in\{0,1,2\}$ the analysis is straightforward, so we begin with $a=3$. 
\begin{enumerate}
	\item $a=3$: The Hilbert function is $(1,4,7,4,1)$ and a computation shows the betti table must be (dropping the $1$ in upper left and lower right corners) 
	\begin {center}
	$ \left[ \!
	\begin{array}{ccc}
		3 & b  &c   \\
		b+4 & 2c+12  &b+4 \\
		c     & b      & 3
	\end{array}\! \right]. $
\end{center}
By Macaulay's theorem
\[
h_2=7 = \binom{4}{2} + \binom{1}{1}, \mbox{ so } h_2^{\langle 2 \rangle} = 11 \ge h_3 = 20 -3 \cdot 4 +b, \mbox{ so } b\le 3.
\]
A direct computation shows that for an ideal generated by three quadratic monomials in $R$, $v \in \{(0,0), (1,0), (2,0), (3,1)\}$, all of which occur in Tables 1 and 2. By uppersemicontinuity, $I_2$ must have $v=(b,c) \le (b',c')$ for $(b',c')$ in the list above, so we need only show that $v \in \{(3,0), (2,1) \}$ do not occur. If $b=3$ then we are in the situation where Gotzmann's theorem applies, and we compute 
\[
h_3^{\langle 3 \rangle}=h_4 = 16 = 35-3\cdot 10 + 3 \cdot 4 -c + b_{24}(I_2). 
\]
In particular, $c \ge 1+b_{24}(I_2)$, so $c \ge 1$ and $v=(3,0)$ does not occur. By Lemma \ref{lem_v_neq_(2,1)}, $v=(2,1)$ is impossible. When $a \ge 4$, the set of betti tables possible for quadratic monomial ideals has an element that is so large that a similar analysis via the initial ideal becomes cumbersome.
\vskip .05in
\item $a=4$: The Hilbert function is $(1,4,6,4,1)$ and the betti table is:
\begin {center}
$ \left[ \!
\begin{array}{ccc}
	4 & b  &c   \\
	b & 2c+6  &b \\
	c     & b      & 4
\end{array}\! \right]. $
\end{center}
Values for $v$ which actually occur are $v \in \{(0,0), (2,0), (3,0), (4,1)\}$. Applying Macaulay's theorem to the ideal $I_2$ generated by the quadrics in $I$ shows $b \le 6$. Now let $C_3$ denote the ideal generated by the quadrics and cubics in $I$. 
\[
h_3(C_3) = h_3(I) = 4=\binom{4}{3} \mbox{ so } h_3^{\langle 3 \rangle}(C_3)=5 \ge h_4(C_3) = c+1.
\]
Hence $c \le 4$. The case $b=6$ is extremal, and applying Gotzmann's theorem we find
\[
h_4(I_2) =35-4\cdot 10+6 \cdot 4 +b_{24}(I_2) -c =15, \mbox{ so } c=4+b_{24}(I_2).
\]
Combined with our work above, this shows $b=6 \Rightarrow c=4$. As $h_4(C_3) = h_4(I)+4=5$, we have
\[
h_4^{\langle 4 \rangle}(C_3) = 6 \ge h_5(C_3) = 56-80+40+b_{25}(C_3)-6,
\]
we conclude $b_{25}(C_3)\le -4$, which is impossible. Thus, $b \in \{0, \ldots,5\}$. If $b \in \{0,1\}$ then $c=0$; clearly $v=(0,0)$ yields a complete intersection, which occurs, while $v=(1,0)$ leads to an almost complete intersection (ACI), and by \cite{Ku} there are no Gorenstein ACI's.  Henceforth we assume $b \in \{2,3,4,5\}$. 
We saw above that $c \le 4$; we now show that $c \in \{2,3,4\}$ is IGH. 
\[
h_5(C_3) = 56-80+4(c+6)+b_{25}(C_3)-b.
\]
So
\[
\begin{array}{ccc}
c=2 \implies h_4(C_3) =3 \implies h_4^{\langle 4 \rangle} (C_3) =3 & \ge & h_5(C_3) = 8+b_{25}(C_3)-b\\
c=3 \implies h_4(C_3) =4 \implies  h_4^{\langle 4 \rangle} (C_3) =4 & \ge & h_5(C_3) = 12+b_{25}(C_3)-b\\
c=4 \implies  h_4(C_3) =5 \implies  h_4^{\langle 4 \rangle} (C_3) =6 & \ge & h_5(C_3) = 16+b_{25}(C_3)-b
\end{array}
\]
As $b \le 5$, only the case $b=5, c=2, b_{25}(C_3)=0$ is possible; this has betti table 
\begin {center}
$ \left[ \!
\begin{array}{ccc}
4 & 5  &2   \\
5 & 10  &5 \\
2     & 5      & 4
\end{array}\! \right]. $
\end{center}
Computing, we find that in this situation $h_5(C_3)=3$, so
\[
h_5^{\langle 5 \rangle} (C_3) =3 \ge  h_6(C_3) = 84 - 140+80-20+b_{26}(C_3)-b_{36}(C_3).
\]
In particular, $b_{36}(C_3) \ge 1+b_{26}(C_3)$, which means the $5 \times 4$ submatrix  $M$ of $d_3$ representing  the “bottom right corner” of the table for $I$,
one of the four columns of $M$ is zero. By symmetry of the free resolution this means that one of the four rows of
the matrix $M^t$ of linear first syzygies on $I_2$ is zero. Hence the five linear first syzygies on $I_2$ only involve a subideal $J \subseteq I_2$ generated by 3 quadrics, which is impossible.  \newline

\noindent It remains to deal with $c \in \{0,1\}$. When $c=0$, we know $v \in \{(0,0),(2,0), (3,0)\}$ occur, and we have already shown that $v=(1,0)$ is IGH. As $b \le 5$, we need to show 
$v \in \{(4,0), (5,0)\}$ are IGH. To do this, we use the ideal $I_2$ of four quadrics; $h_3(I_2) = 20-16+b=4+b$, so we have
\[
\begin{array}{ccc}
b=4 \Rightarrow  h_3(I_2)=8 \Rightarrow h_3^{\langle 3 \rangle}(I_2)=10 & \ge h_4(I_2) = 35-40+16+b_{24}(I_2) = 11+ b_{24}(I_2) \\
b=5 \Rightarrow  h_3(I_2)=9 \Rightarrow h_3^{\langle 3 \rangle}(I_2)=12 & \ge h_4(I_2) = 35-40+20+b_{24}(I_2) = 15+ b_{24}(I_2),
\end{array}
\]
both of which force $b_{24}(I_2) \le -1$, which is impossible. When $c=1$, the only change to the second equation above is to subtract one (because $c=1$) from the right hand side, so $h_4(I_2)=14+b_{24}(I_2)$, forcing $b_{24} \le -2$, which is impossible. 


\vskip .15in
\item $a=5$:  The Hilbert function is $(1,4,5,4,1)$ so the betti table is
\begin {center}
$ \left[ \!
\begin{array}{ccc}
5 & b  &c   \\
b-4 & 2c &b-4 \\
c     & b      & 5
\end{array}\! \right]. $
\end{center}
Note that $h_3 = 20-5 \cdot 4 +b$, so $h_3 =b$. By Macaulay's theorem 
\[
h_2=5 = \binom{3}{2} + \binom{2}{1}, \mbox{ so } h_2^{\langle 2 \rangle} = 7 \ge h_3 = b, \mbox{ so }7 \ge b.
\]
If $b=7$, applying Gotzmann's theorem gives $c=b_{24}(I_2)+4$. Let $C_3$ denote the subideal of $I$ generated in degrees two and three; applying Macaulay's theorem to $h_3(C_3)=4$ yields
\[
5 \ge h_4(C_3) =35-5 \cdot 10 +4 \cdot 4 +c,
\]
so $c \le 4$; combined with $c = b_{24}(I_2)+4$ this forces $c=4$. Since $h_4^{\langle 4 \rangle}(C_3)=6$, we find
\[
6 \ge h_5(C_3) = 56-5\cdot 20 + 4\cdot 10 +4 \cdot 4 +b_{25}(C_3)-3 = 9+b_{25}(C_3).
\]
This shows $b_{25}(C_3) \le -3$, hence $b=7$ is IGH, and $b \in \{4,5,6\}$. \newline
\mbox{  }Case 1: Suppose $b=4$. This means there are no cubics in the ideal, and 
\[
h_3=4 = \binom{4}{3} \mbox{ so }h_3^{\langle 3 \rangle} = 5 \ge h_4 = 35-5\cdot 10 + 4 \cdot 4 + c.
\]
We conclude $c \le 4$. We can immediately rule out $c=0$, as then $I$ would be an ACI, which is IGH. The possibilities $c \in \{2,3,4\}$ are also ruled out by Macaulay; we illustrate for $c=2$:
\[
h_4 = 35-5 \cdot 10 + 4 \cdot 4 +2 = 3, \mbox{ so } h_4^{\langle 4 \rangle} = 3 \ge h_5 = 4+b_{25}(I_2),
\]
which would force $b_{25}(I_2) \le -1$. Finally, suppose $c=1$, so $I=I_2+q$ for a single quartic $q$. Since $I_2+q$ has codimension four, the codimension of $I_2$ must be three or four, and if $\codim(I_2)=4$ then $I_2$ contains a complete intersection $C$. We claim this is impossible: write $I_2 = C+f$ with $f \in I_2 \setminus C$. Since $b_{23}(C)=0$ the fact that $b_{23}(I_2)=4$ means that $C:f=\langle x_1,x_2,x_3,x_4\rangle$, whose mapping cone is inconsistent with the betti table for $I_2$. Hence $\codim(I_2)=3$, and $q$ is a nonzero divisor on the codimension three associated primes of $I_2$. Since $h_4(I_2)=2$, Macaulay's theorem implies the degree of $I_2$ is one or two. Observe that the rank of the linear second syzygy $S$ cannot be 4; if it was then $S=[x_1,x_2,x_3,x_4]^t$. By the symmetry of the differentials in the free resolution, this means that $I_2:q=\langle x_1,\ldots, x_4\rangle$. By additivity of the Hilbert polynomials on the short exact sequence
\[
0 \longrightarrow R(-4)/(I_2:q) \longrightarrow R/I_2 \longrightarrow R/I \longrightarrow 0,
\]
this is impossible. Hence $\rank(S) =3$, and as in the proof that $v=(3,1)$ is impossible for $a =4$, $I_2$ must contain, after a change of variables,  $\{L\cdot x_1, L \cdot x_2, L\cdot x_3\}$ for a linear form $L$.
Since $\codim(I_2)=3$, this forces $L, q_4,q_5$ to be a regular sequence. In particular, $\deg(I_2)=4$, a contradiction. \newline
\mbox{  }Case 2: Suppose $b=5$. The cases $v \in \{(5,0),(5,1)\}$ do occur. 
\[
h_3=5 = \binom{4}{3} + \binom{2}{2} \mbox{ so }h_3^{\langle 3 \rangle} = 6 \ge h_4 = 35-5\cdot 10 + 4 \cdot 5 + b_{24}(I_2)-c.
\]
So $c+1 \ge b_{24}(I_2)$. Let $C_3$ denote the subideal of $I$ generated in degrees two and three.
\[
h_3(C_3)=4\mbox{ so }h_3^{\langle 3 \rangle}=5, \mbox{ thus } 5 \ge h_4 = 35-50+16+c,
\]
implying $c \le 4$. Since $c \in \{0,1\}$ does occur, we need to rule out $c \in \{2,3,4\}$. Computing values for $h_4$, we find
\[
\begin{array}{ccc}
c=2\mbox{ implies }&h_4 =3 &\mbox{ hence }h_5 \le 3\\
c=3\mbox{ implies }&h_4 =4 &\mbox{ hence }h_5 \le 4\\
c=4\mbox{ implies }&h_4 =5 &\mbox{ hence }h_5 \le 6
\end{array}
\]
Since $h_5 = 56-100+40+4c-1+b_{25}(C_3)$, combining this with the above shows
\[
\begin{array}{ccc}
c=2\mbox{ implies }&h_5 =3+b_{25}(C_3) & \le 3\\
c=3\mbox{ implies }&h_5 =7+b_{25}(C_3) & \le 3\\
c=4\mbox{ implies }&h_5 =11+b_{25}(C_3) & \le 6
\end{array}
\]
This rules out $c \in \{3,4\}$, and shows if $c=2$ then $b_{25}(C_3)=0$. So in this case $h_5(C_3)=3$,  and
\[
h_5^{\langle 5 \rangle}(C_3) = 3 \ge h_6(C_3) = 84-175+80+20-4+b_{26}(C_3)-b_{36}(C_3).
\]
In particular, we have $3 \ge 5 +b_{26}(C_3)-b_{36}(C_3)$, hence $b_{36}(C_3) \ge 2$, so the betti table for $C_3$ is at least
\begin {center}
$ \left[ \!
\begin{array}{ccc}
5 & 5  &2   \\
1 & 4 &1 \\
0     & 0     & 2
\end{array}\! \right]. $
\end{center}
Hence in the $5 \times 5$ submatrix  $M$ of $d_3$ representing  the “bottom right corner” of the table for $I$,
two of the five columns of $M$ are zero, which by symmetry of the betti table means that two of the five rows of
the matrix $M^t$ of linear first syzygies on $I_2$ are zero. Hence the five linear first syzygies on $I_2$ only involve a subideal $J \subseteq I_2$ generated by 3 quadrics, which is impossible.  \newline
\mbox{  }Case 3: Suppose $b=6$; the only case that actually occurs is $v=(6,2)$. 
\[
h_3=6 = \binom{4}{3}+ \binom{2}{2} +\binom{1}{1} \mbox{ so }h_3^{\langle 3 \rangle} = 7 \ge h_4 = 35-5\cdot 10 + 4 \cdot 6 + b_{24}(I_2)-c.
\]
So $c  \ge b_{24}(I_2)+2$.  Let $C_3$ denote the subideal of $I$ generated in degrees two and three.
\[
h_3(C_3) = 4 = \binom{4}{3} \mbox{ so }5 \ge h_4 = 35-50+16+c
\]
Thus, $c \le 4$. To show that $c \in \{3,4\}$ do not occur, we compute
\[
\begin{array}{ccc}
\mbox{If }c=4, \mbox{ then }&h_4(C_3)=5 \mbox{ and }&h_5(C_3) \le 6 \\
\mbox{If }c=3, \mbox{ then }&h_4(C_3)=4 \mbox{ and }&h_5(C_3) \le 4 
\end{array}
\]
Since $h_5(C_3)=56-100+40+4c +b_{25}(C_3)-2$, we see that 
\[
\begin{array}{ccc}
\mbox{If }c=4&\mbox{ then } h_5 = 10+b_{25}& \le 6, \mbox{ so } b_{25}(C_3) \le -4\\
\mbox{If }c=3&\mbox{ then } h_5 = 6+b_{25}& \le 4, \mbox{ so } b_{25}(C_3) \le -2
\end{array}
\]
We have shown that when $b=6$, the only value possible for $v$ is $(6,2)$.\newline
\vskip .1in
\item $a=6$:  The Hilbert function is $(1,4,4,4,1)$ so the betti table is
\begin {center}
$ \left[ \!
\begin{array}{ccc}
6 & b  &c   \\
b-8 & 2c-6  &b-8 \\
c     & b      & 6
\end{array}\! \right]. $
\end{center}
As 
\[
h_2=4 = \binom{3}{2} + \binom{1}{1}, \mbox{ Macaulay's theorem shows } h_2^{\langle 2 \rangle} = 5 \ge h_3 = 20 -6 \cdot 4 +b.
\]
So $b \le 9$. If $b=9$ there is a unique cubic $F \in I$; since $b=9$ is extremal we may apply Gotzmann's theorem, yielding $b_{24}(I_2)=c-5$. Since $(2c-6) -(c-5)=c-1$ and $c \ge 5$, this means there are always at least four independent syzygies which are linear on $F$ and quadratic on elements of $I_2$.  Hence $I_2 :F=\langle x_1,\ldots x_4\rangle$ and the mapping cone arising from short exact sequence 
\[
0\longrightarrow S(-3)/I_2:F \longrightarrow S/I\longrightarrow S/I_2+F \longrightarrow 0,
\]
gives a resolution of $S/I$. The top row of the mapping cone is simply the Koszul complex on the variables, and a check of the degrees shows the second syzygies involve a summand $R^6(-5)$ which cannot cancel. This would imply $b_{35}(I)=b-8 \ge 6$, which is impossible since $b=9$. 

Finally, we need to show that when $b=8$ we must have $c=3$. From the Hilbert function constraint on the betti table, $c \ge 3$.  When $b=8$, there are no cubics in $I$; this means 
\[
b_{24}(I_2)-c = c-6.
\] 
We compute 
\[
h_3=4=\binom{4}{3}\mbox{ so }h_3^{\langle 3 \rangle} = 5 \ge h_4 = 35 -6 \cdot 10 +8\cdot 4 +c-6,
\]
hence $c \le 4$. Finally, if $c=4$, then $h_4=5$ and $h_4^{\langle 4 \rangle }= 6$. So 
\[
6 \ge h_5 = 56-6\cdot 20+8\cdot10-2\cdot4+b_{25}(I_2).
\]
This would force $b_{25}(I_2) \le -2$. We have shown that the only betti table possible for $a=6$ is
\begin {center}
$ \left[ \!
\begin{array}{ccc}
6 & 8  &3   \\
0 & 0  &0 \\
3     & 8      & 6
\end{array}\! \right]. $
\end{center}
\end{enumerate}
Hence there are 16 betti tables for an Artin Gorenstein algebra $A$ with $\reg(A)=4 = \codim(A)$. All diagrams in Table 1 and Table 2 do occur, which can be checked via a {\tt Macaulay2} search. \end{proof}

\section{Gorenstein deviation two} 
The {\em deviation} of an ideal $I$ is the number of generators of
$I$ minus the codimension of $I$. Complete intersections are the simplest 
Gorenstein rings, and have deviation zero; in \cite{Ku}, Kunz shows a Gorenstein ring cannot have deviation one.
In this section, we study Gorenstein rings of deviation two. This is similar to the codimension three case, where the classification of \cite{BE} shows that such ideals come from Pfaffians.
In \cite{hu}, Huneke-Ulrich give a construction for Gorenstein rings
of deviation two; Let $Y$ be a $2n \times 2n$ skew
symmetric matrix of variables, and $X$ a $1 \times 2n$ vector of variables. Then
the ideal generated by the quadrics in $Y\cdot X$ plus the Pfaffian of
$Y$ is Gorenstein deviation two. Such an ideal will have regularity
four iff $n=3$, and we analyze this case in \S3.2. The corresponding 
GoCY threefold $X \subseteq \p^8$ has Hodge numbers different 
from the $h^{p,q}(X)$ for any $X \subseteq \p^n$ with $n \le 7$. 

By the Buchsbaum-Eisenbud theorem, the Pfaffians of a skew $5
\times 5$ matrix $M$ also have deviation two, and quotienting such an
ideal by a regular sequence preserves the deviation two
property. If $M$ is a matrix of linear forms, in order to have regularity four, the regular sequence must consist of
two quadrics or a single cubic; if $M$ has linear and quadratic entries, the regular sequence is a single quadric. 
We analyze these ideals in \S 3.3. 

\subsection{Huneke-Ulrich ideals}
\vskip .1in
\noindent Let $I$ be the ideal consisting of the six quadrics of $Y\cdot X$
\[
\left[ \!
\begin{array}{cccccc}
0 & -y_{12} & \cdots & \cdots& \cdots &-y_{16}\\
y_{12} &0 &-y_{23} & \cdots & \cdots &\cdots\\
\cdots &y_{23}&0 & -y_{34} & \cdots&\cdots\\
\cdots & \cdots& y_{34} &0 &-y_{45}&\cdots\\
\cdots & \cdots& \cdots       &y_{45} &0 &-y_{56}\\
y_{16} & \cdots& \cdots &\cdots& y_{56}&0
\end{array}\! \right]
\cdot
\left[ \!
\begin{array}{c}
x_1\\
x_2\\
x_3\\
x_4\\
x_5\\
x_6
\end{array}\! \right]=0
\]
along with the cubic Pfaffian of $Y$. By \cite{hu}, $I$ is
Gorenstein, with $V(I)$ of dimension 15 in $\p^{20}$, and is
singular in dimension 8. Therefore quotienting by a regular
sequence of 12 linear forms yields a smooth GoCY threefold in $\p^8$
with betti diagram identical to that of $I$:
\begin{small}
\[
\vbox{\offinterlineskip 
\halign{\strut\hfil# \ \vrule\quad&# \ &# \ &# \ &# \ &# \ &# \
&# \ &# \ &# \ &# \ &# \ &# \ &# \
\cr 
total&1&7&22&22&7&1\cr 
\noalign {\hrule}
0&1 &--&--& --&--& --&   \cr 
1&--&6 &1 & --&--&--&        \cr 
2&--&1 &21 & 21&1&--&        \cr 
3&--&-- &-- &1&6&--&         \cr 
4&--&-- &-- &--&--&1&         \cr 
\noalign{\bigskip}
\noalign{\smallskip}
}}
\]
\end{small}
\noindent This yields a GoCY threefold in $\p^8$, whose Artinian reduction has $H$-vector $(1,5,9,5,1)$. So $\deg(X)=21$, and a {\tt Macaulay2} computation shows that for a generic $X$ the Hodge numbers are
\[
\begin{array}{ccccccc}
&&&1&&& \\
&&0&&0&& \\
&0&&1&&0& \\
1  &&52&&52&&1 \\
&0&&1&&0& \\
&&0&&0&& \\
&&&1&&&
\end{array}
\]
\noindent These are not the $h^{p,q}(X)$ for any GoCY $X$ of codimension $\le 4$, and correspond to a point on the line connecting the Hodge numbers of the degree 13 and 14 examples of Tonoli--see Figure 1 in \cite{T}. While $X$ can be projected to a smooth CY in $\p^7$, it is no longer projectively normal (hence not a GoCY). 
\subsection{Quotients of Pfaffians, I}
A second class of Gorenstein ideals of deviation two
and regularity four come from quotienting the Pfaffians $Pf(M)$ of a $5
\times 5$ matrix $M$ of linear forms by  a regular sequence.  If $M$ is
generic, then $Pf(M)$ is codimension 3 in $\p^9$, and is smooth. Quotienting by a regular sequence of 
degrees $\{ 2,2,1 \}$  yields a smooth GoCY threefold $ X \subseteq \p^8$ of degree 20, with betti table
\begin{small}
$$
\vbox{\offinterlineskip 
\halign{\strut\hfil# \ \vrule\quad&# \ &# \ &# \ &# \ &# \ &# \
&# \ &# \ &# \ &# \ &# \ &# \ &# \
\cr 
total&1&7&16&16&7&1\cr 
\noalign {\hrule}
0&1 &--&--& --&--& --&   \cr 
1&--&7 &5 & --&--&--&        \cr 
2&--&-- &11 & 11&--&--&        \cr 
3&--&-- &-- &5&7&--&         \cr 
4&--&-- &-- &--&--&1&         \cr 
\noalign{\bigskip}
\noalign{\smallskip}
}}
$$
\end{small}
\noindent For the generic case, $h^{1,1}(X)=1$ and $h^{1,2}(X)=61$, which are attained by Tonoli's degree 13 example. On the other hand, if we quotient the pfaffians of $M$ by a generic cubic, this yields an ideal $I$ with betti table CGKK 2.
Quotienting with two generic linear forms yields a smooth GoCY $X \subseteq \p^7$ of degree $15$, discovered in \cite{vEvS}, which has $h^{1,1}(X)=1$ and $h^{1,2}(X)=76$. 

\subsection{Quotients of Pfaffians, II}
We now show that the betti diagram of Type 2.4 corresponds to a mapping cone, and that any nondegenerate irreducible GoCY in $\p^7$ with betti table of Type 2.4 must be singular. Recall Type 2.4 has betti diagram 
\begin{small}
$$
\vbox{\offinterlineskip 
\halign{\strut\hfil# \ \vrule\quad&# \ &# \ &# \ &# \ &# \ &# \
&# \ &# \ &# \ &# \ &# \ &# \ &# \
\cr 
total&1&6&10&6&1\cr 
\noalign {\hrule}
0&1&-- &--&--&-- &\cr 
1&--&4&2 &--& --  &     \cr 
2&--&2&6 &2 &--  &    \cr 
3&--&--&2 &4&--     &   \cr 
4&--&--&--&-- &1      &   \cr 
\noalign{\bigskip}
\noalign{\smallskip}
}}
$$
\end{small}\vskip -.2in
A key tool in our analysis is a result of Vasconcelos-Villereal \cite{VV}, which shows that if $R$ is a Gorenstein local ring and $2\in R$ is a unit, then if $I$ is a Gorenstein ideal of codimension $4$ and deviation two, such that $I$ is a generic complete intersection (the localization at all minimal primes is a complete intersection), then $I$ is a hypersurface section of a Gorenstein ideal of height $3$. Table 1 and Table two contain two examples of Gorenstein ideals of codimension four and deviation two: CGKK 2, and Type 2.4; the first can be obtained by quotienting the Pfaffians of a skew matrix of linear forms by a cubic. 
\pagebreak

\begin{thm} A nondegenerate irreducible threefold in $\p^7$ with betti diagram of Type 2.4 is singular.
\end{thm}
\noindent We start with several preparatory lemmas. Note that a betti diagram of Type 2.4 cannot arise as the mapping cone of a cubic, so will arise from quotienting the Pfaffians by a quadric. 
\begin{lem}\label{quadricsinI}
There is a prime subideal $J \subseteq I_2$ generated by three quadrics, such that $J$ consists of the $2 \times 2$ minors of a 1-generic $2 \times 3$ matrix $M$, and the quadric $q_4 \in I_2\setminus J$ is a nonzero divisor on $R/J$. 
\end{lem}
\begin{proof}
By Theorem 1.7 of \cite{ss}, a linear first syzygy on $I_2$ of rank four would imply that $I_2$ contains the Pfaffians of a $5 \times 5$ skew matrix of linear forms, while if there was a linear first syzygy on $I_2$ of rank two, $I$ would not be prime. So Theorem 1.7 implies that $I_2$ contains a subideal $J$ of  $2 \times 2$ minors of a 1-generic $2 \times 3$ matrix of linear forms. The ideal $J$ must be prime, for if not, it would have a primary decomposition into components of degrees one or two, which would force $I$ to be degenerate. Finally, $q_4$ is regular on $J$, for if not, then $\codim(I_2)=2$ and degree one or two; the two cubics in $I$ must be nonzero divisors on the codimension two primary component, because $\codim(I)=4$. But this would imply that $\deg(I)$ is $9$ or $18$, contradicting the fact that $\deg(I)=16$. 
\end{proof}
\noindent In what follows, we use the notation of Lemma ~\ref{quadricsinI}, so $J$ is the ideal of $2 \times 2$ minors of the one-generic matrix $M$. The entries of $M$ are linear forms, because $J$ is prime the linear forms span a space of dimension $\{4,5,6\}$. This means $V(J)$ is a cone, with singular locus of dimension (respectively) $\{3,2,1\}$. Let $C$ be the ideal generated by $q_4$ and the two cubic generators of $I$; intersecting $V(J)$ with $V(C)$ drops the dimension by two, so if the linear forms of $M$ span a space of dimension four or five, $V(I)$ is singular. It remains to deal with the case that the span of the linear forms has dimension six; after a change of variables we may assume 
\begin{align*}
M=\begin{bmatrix}
x_1 & x_2 & x_3\\
x_4 & x_5 & x_6
\end{bmatrix}
\end{align*}

\begin{lem}\label{vvapplication}
Let $I$ be a codimension four Gorenstein prime ideal with betti diagram Type 2.4. If $I_2$ contains an ideal $J$ consisting of the $2 \times 2$ minors of $M$ as above, then $I=I'+\langle F \rangle,$ with $\codim(I')=3$ and $I'$ Gorenstein, and $F$ a nonzero divisor on $R/I'$. Hence $R/I$ has a mapping cone resolution.
\end{lem}
\begin{proof}
Because the two linear first syzygies on $I_2$ are of the form $[x_1,x_2,x_3]^t$ and $[x_4,x_5,x_6]^t$ and $I$ is nondegenerate, $I$ contains no linear form, so $\{x_1,\ldots,x_6\}$ are all units when $R/I$ is localized at $I$. Thus, in the localization, two of the generators for $J$ are redundant, and therefore $I$ is a generic complete intersection, of deviation two, so the result of \cite{VV} applies (we assume henceforth that $2$ is a unit).
\end{proof}

\begin{lem}\label{3223betti}
Assume $Y$ is an arithmetically Gorenstein variety of codimension $3$ and $X$ is a nondegenerate hypersurface section of $Y$ with betti diagram of Type 2.4. Then $Y$ must have betti diagram: \end{lem}
\begin{small}
$$\vbox{\offinterlineskip 
\halign{\strut\hfil# \ \vrule\quad&# \ &# \ &# \ &# \ &# \ &# \
&# \ &# \ &# \ &# \ &# \ &# \ &# \
\cr 
total&1&5&5&1\cr 
\noalign {\hrule}
0&1&-- &--&--&\cr 
1&--&3&2 &--&     \cr 
2&--&2&3 &-- &   \cr 
3&--&--&-- &1&   \cr 
}}
$$
\end{small}
\begin{proof}
The Hilbert series of $X$ is 
\begin{align*}
h_{t}(X)=\frac{1}{(1-t)^{n}}(1-t^2)^4.
\end{align*}
Assume $h_{t}(Y)=\frac{1}{(1-t)^n}f(t)$. Then 
\[
f(t)(1-t^{d})=(1-t^2)^4. 
\]
So $d\in\{1,2\}$. But $X$ does not lie in any hyperplane. Therefore $d$ must be $2$ and $Y$ has the desired betti table.
\end{proof}

\begin{prop}\label{bettiSkew3223}
Let $V(I)$ be GoCY in $\p^7$ with betti diagram of Type 2.4. If the linear forms of the matrix $M$ span a space of dimension six, then up to a change of basis, $I$ is generated by the Pfaffians of a $5\times5$ skew symmetric matrix $N$ as below, along with a quadric $q_4$ which is a nonzero divisor on $R/\Pfaff(N)$. The ideal $\Pfaff(N)$ is singular along a $\p^1$, and so $V(I)$ has at least two singular points. 
\begin{align*}
N=\begin{bmatrix}
0& x_1 & x_2 & x_3& 0\\
-x_1 & 0 & q_1 & q_2 & x_4\\
-x_2 & -q_1 & 0 & q_3 & x_5\\
-x_3 & -q_2 & -q_3 & 0 & x_6\\
0 & -x_4 & -x_5 & -x_6 & 0
\end{bmatrix}
\end{align*}
where the $q_j$'s are quadrics.
\end{prop}
\begin{proof}
Combining Lemmas~\ref{quadricsinI}, \ref{vvapplication}, and \ref{3223betti} and the results of \cite{VV} shows that $I$ is of the form above. To see that the singular locus is as claimed, we compute that 
\[
\Pfaff(N) = J + \langle x_3q_1-x_2q_2+x_1q_3, x_6q_1-x_5q_2+x_4q_3 \rangle,
\]
where $J$ is the ideal of the minors of the matrix $M$ above Lemma~\ref{vvapplication}. In particular, 
\[
V(x_1, \ldots x_6) \simeq \p^1 \subseteq V(\Pfaff(N)),
\]
and $V(\Pfaff(N))$ is singular along this $\p^1$, because the Jacobian matrix of $\Pfaff(N)$ is
\begin{align*}
\mathrm{Jac}(\Pfaff(N))= \begin{bmatrix}
x_5 & x_6 & 0 & * & *\\
-x_4 & 0 & x_6 & * & *\\
0 & -x_4 & -x_5 & * & *\\
-x_2 & -x_3 & 0 & * & *\\
x_1 & 0 & -x_3 & * & *\\
0 & x_1 & x_2 & * & *\\
0 & 0 & 0 & * & *\\
0 & 0 & 0 & * & *
\end{bmatrix}
\end{align*}
where $*$ are quadrics. Hence when $\{x_1,\ldots,x_6\}$ vanish, $\mathrm{Jac}(\Pfaff(N))$ has rank $\leq 2$, so is singular along the $\p^1$. Intersecting $V(\Pfaff(N))$ with the hypersurface $V(q_4)$, we find that $V(I)$ must be singular (at least) at a degree two zero scheme. 
\end{proof}


\section{Computational aspects and Ideals with mostly quadratic generators}
As noted earlier, GoCY's were first investigated systemically from a computational standpoint by Bertin in \cite{Ber}. Below we describe algorithms which, in certain situations, offer a substantial speed up in processing; in some cases, we have seen an improvement in runtime by a factor of 500. 

Let $I$ be the ideal sheaf of a smooth Calabi-Yau threefold in $\p^n$.  This implies 
that 
\[ H^1(\mathcal{O}_X) = H^2(\mathcal{O}_X) = 0,\]
and Kodaira vanishing
and Serre duality give that 
\[ H^1(\mathcal{O}_X(-1)) = H^2(\mathcal{O}_X(-1)) = 0. \]
From the fundamental short exact sequence
\begin{equation}\label{cotan}
0 \longrightarrow I/I^2 \longrightarrow \Omega^1_{\p^n}\otimes \mathcal{O}_X \longrightarrow \Omega^1_X \longrightarrow 0,
\end{equation}
we have that 
\[
\chi(\Omega^1_X) = \chi(\Omega^1_{\p^n}\otimes \mathcal{O}_X)-\chi(I/I^2).
\] 
The Euler characteristic of both sheaves can be computed from the corresponding modules of global sections via
Gr\"obner bases; it turns out we actually only need one computation. Tensoring the short exact sequence 
\begin{equation}\label{cotan2}
0 \longrightarrow \Omega^1_{\p^n} \longrightarrow \mathcal{O}^{n+1}_{\p^n}(-1) \longrightarrow \mathcal{O}_{\p^n} \longrightarrow 0
\end{equation}
with $\mathcal{O}_X$, and using the cohomology vanishings for $H^1(\mathcal{O}_X)$ and $H^2(\mathcal{O}_X)$ yields that 
\[
\chi(\Omega^1_{\p^n}\otimes \mathcal{O}_X) = (n+1)HP(S/I_X,-1).
\]
In particular, all the computational expense to compute $\chi(\Omega^1_X)$ comes from computing $in(I/I^2)$. To compute $h^{1,1}(X) $, we use the long exact sequence in cohomology of Equation~\ref{cotan}. Vanishing of $H^0(\Omega^1_X)$ due to the Calabi-Yau property and the long exact sequence in cohomology yield
\[
0 \longrightarrow H^1(I/I^2)
\longrightarrow H^1(\Omega^1_{\p^n}\otimes \mathcal{O}_X)\longrightarrow H^1(\Omega_X)\longrightarrow H^2(I/I^2)\longrightarrow H^2(\Omega^1_{\p^n}\otimes \mathcal{O}_X) \longrightarrow \cdots
\]
\noindent Now, $H^2(\Omega^1_{\p^n}\otimes \mathcal{O}_X)=0$, 
because it sits in the exact sequence 
\[
\cdots \longrightarrow H^1(\mathcal{O}_X)\longrightarrow H^2(\Omega^1_{\p^n}\otimes \mathcal{O}_X) \longrightarrow 
H^2(\mathcal{O}_X^{n+1}(-1)) \longrightarrow \cdots, 
\]
and from Equation~\ref{cotan2} we have $h^1(\Omega^1_{\p^n}\otimes \mathcal{O}_X) = 1$. Hence
\[ h^{1,1}(X) = h^2(I/I^2) - h^1(I/I^2)+1. \]
Writing $I/I^2$ for both the graded $S$-module and the sheaf, by local duality (see \cite{E})
\begin{equation}\label{Extvanish}
h^1(I/I^2) = \dim Ext^{n-1}(I/I^2,S)_{-n-1} \mbox{ and }  h^2(I/I^2)=\dim Ext^{n-2}(I/I^2,S)_{-n-1} 
\end{equation}
If the projective dimension of the $S$-module $I/I^2$ is less than $n-2$, then the vanishing of the Ext modules above is automatic, and $h^{1,1}(X) = 1$. If not, we can compute $h^{1,1}(X)$ from the formula. Projective dimension can be computed quickly using the {\tt Macaulay2} command {\tt minimalBetti}, which we illustrate below.
\begin{exm}
We revisit the Huneke-Ulrich ideal appearing in \S 3,
defining a smooth Calabi-Yau 3-fold in $\p^8$ of degree 21. 
\begin{verbatim}

-- Input: a homogeneous ideal I defining a dim 3 smooth Calabi-Yau variety X
-- Output: Euler characteristic of Omega^1_X

CYEuler = (I) -> (
M := coker prune presentation (I/I^2);
MM := coker leadTerm gb presentation M;
HP1 := hilbertPolynomial MM;
F := source vars ring I;
HP2 := hilbertPolynomial(F ** coker gens I);
euler HP2 - euler HP1)

i47 : time CYEuler HunekeUlrich
-- used 28.3018 seconds

o47 = 51

i48 : time minimalBetti M          
-- used 2.71652 seconds

0  1   2   3  4  5
o48 = total: 7 50 140 154 78 21
2: 6  1   .   .  .  .
3: 1 42  27   1  .  .
4: .  6 107 132 21  .
5: .  1   6  21 57 21

\end{verbatim}
By Equation~\ref{Extvanish}, $h^{1,1}(X) = 1$, because $Ext^{n-2}(I/I^2,S) = 0$, so in particular $Ext^{n-2}(I/I^2,S)_{-n-1} = 0$.   Since the Euler characteristic is 51, $h^{1,2}(X) = 52$.
\end{exm}

\pagebreak  

\noindent{\bf Remarks and Open Questions}: 
\begin{enumerate}
\item For codimension four and regularity two, only one betti table occurs: for the $2 \times 2$ minors of a generic $3 \times 3$ matrix, and for regularity three, there are three betti tables. The calculations are the same as in the proof for Theorem 2.2, and we leave them to the interested reader. 
\item For regularity four and codimension five, and regularity five and codimension four there are, respectively, 68 and 36 possibilities, which can be tackled with the same approach.\item Relate Theorem~\ref{possibleBetti} to the parameter space $Gor_T$, studied in depth in \cite{ik}.
\item Theta characteristics are Gorenstein of regularity 4. Is there a connection to GoCY's?
\item Gorenstein point configurations \cite{EP} have Gale duality. Does this lead to a mirror construction?

\end{enumerate}
\renewcommand{\baselinestretch}{1.0}
\small\normalsize 

\bibliographystyle{amsalpha}

\begin{thebibliography}{10}

\bibitem{B} V. Batyrev, 
Dual polyhedra and mirror symmetry for Calabi-Yau hypersurfaces in toric varieties,
J. Algebraic Geom. 3 (1994), 493-535. 

\bibitem{Ber} M. Bertin,
Examples of Calabi–Yau 3-folds of $\p^7$ with $\rho = 1$,
Canad. J. Math. 61, (2009), 1050-1072.

\bibitem{BG} G. Brown, K. Georgiadis, 
Polarized Calabi-Yau 3-folds in codimension 4,
Math. Nachr. 290 (2017), 710-725.

\bibitem{BKZ} G. Brown, A. Kasprzyk, L. Zhu,
Gorenstein formats, Canonical and Calabi-Yau threefolds,
Experimental Mathematics (2019) 46–74. 

\bibitem{BE} D. Buchsbaum, D. Eisenbud, 
Algebra structures for finite free resolutions, and some structure theorems for ideals of codimension three, 
Amer. J. Math. 99 (1977),  447-485. 

\bibitem{chsw} P. Candelas, G. Horowitz, A. Strominger, E. Witten, 
Vacuum configurations for superstrings, 
Nuclear Physics B. 258 (1985)  46–74. 

\bibitem{cogp} P. Candelas, X. de la Ossa, P. Green, L. Parkes, 
A pair of Calabi–Yau manifolds as an exactly soluble superconformal field theory,
Nuclear Physics B. 359 (1991): 21–74. 

\bibitem{cls} P. Candelas, M. Lynker, R. Schimmrigk, 
Calabi–Yau manifolds in weighted $\p^4$,
Nuclear Physics B. 341 (1990), 383–402. 

\bibitem{cgkk} S. Coughlan, L. Golebiowski, G. Kapustka, M. Kapustka,
Arithmetically Gorenstein Calabi-Yau threefolds in $\p^7$,
Electron. Res. Announc. Math. Sci. 23 (2016), 52--68.

\bibitem{CK} D. Cox, S. Katz, 
Mirror Symmetry and Algebraic Geometry, 
AMS Math Surveys and Monographs, (1999).

\bibitem{cox} D. Cox, J. Little, H. Schenck, 
Toric Varieties, 
AMS Graduate Studies in Mathematics, (2010).

\bibitem{E} D. ~Eisenbud, 
Commutative Algebra with a view towards Algebraic Geometry, 
Graduate Texts in Mathematics, vol.~150, 
Springer, Berlin-Heidelberg-New York, 1995.

\bibitem{e3} D. ~Eisenbud, 
The geometry of syzygies,
Graduate Texts in Mathematics, vol.~229, 
Springer, Berlin-Heidelberg-New York, 2005.

\bibitem{EP} D. Eisenbud, S. Popescu, 
The projective geometry of the Gale transform,
J. Algebra 230 (2000), 127-173.

\bibitem{em} J. Emsalem, A. Iarrobino,
Some zero dimensional generic singularities: finite algebras having small tangent space,
Compositio Math. 36 (1978), 145-188. 

\bibitem{danmike}D.~Grayson, M.~Stillman,
Macaulay 2: a software system for algebraic
geometry and commutative algebra,
{\tt http://www.math.uiuc.edu/Macaulay2}

\bibitem{GHL} P. Green, T. H\"ubsch, C. L\"utken,
All the Hodge numbers for all Calabi-Yau complete intersections,
Class. Quantum Grav. 6 (1989), 105-124.

\bibitem{GN} T. Gulliksen, O. Neg\.ard, Un complexe
r\'esolvant pour certains id\'eaux d\'eterminantiels,
C. R. Acad. Sci. Paris 274, (1972), 16-18.

\bibitem{YHH1} Y. H. He,
The Calabi-Yau landscape: from geometry to physics to machine learning,
preprint, 2018, {\tt arXiv 1812.02893v1}        

\bibitem{hu} C. Huneke, B. Ulrich,
Divisor class groups and deformations,
Amer. J. Math. 107 (1985), 1265-1303.

\bibitem{ik} T. Iarrobino, V. Kanev, Power sums, Gorenstein algebras, and determinantal loci,
Lecture Notes in Mathematics, 1721. Springer-Verlag, Berlin, 1999.

\bibitem{KK} Kapustka, M., Kapustka, G.,
A cascade of determinantal Calabi-Yau threefolds,
Math. Nachr. 283 (2010), 1795-1809.

\bibitem{Ku} E. Kunz,
Almost complete intersections are not Gorenstein rings,
J. Algebra 28 (1974), 111-115.

\bibitem{MSS} M. Mastroeni, H. Schenck, M. Stillman,
Quadratic Gorenstein rings and the Koszul property II,
preprint, {\tt arXiv 1903.08273}    

\bibitem{MN} J. Migliore, U. Nagel,
Gorenstein algebras presented by quadrics,
Collect. Math. 64 (2013), 211-233.

\bibitem{MZ} J. Migliore, F. Zanello,
Stanley's nonunimodal Gorenstein h-vector is optimal.,
Proc. Amer. Math. Soc. 145 (2017), 1-9.

\bibitem{reid} M. Reid,
Gorenstein in codimension 4: the general structure theory
Adv. Stud. Pure Math, 65 (2015), 201-227.

\bibitem{r} E. R\o dland, 
The Pfaffian Calabi-Yau, its mirror, and their link to the Grassmannian G(2,7), 
Compositio Math. 122 (2000), 135-149. 

\bibitem{sch} H. ~Schenck,
Computational Algebraic Geometry,
Cambridge Univ. Press, Cambridge, (2003).

\bibitem{ss} H. ~Schenck, M.~Stillman,
High rank linear syzygies on low rank quadrics,
American Journal of Mathematics, 134 (2012), 561-579.

\bibitem{T} F. ~Tonoli,
Construction of Calabi-Yau 3-folds in $\p^6$,
J. Algebraic Geom. 13 (2004),  209-232.

\bibitem{vEvS} C. van Enckevort, D. van Straten,
Monodromy calculations of fourth order equations of Calabi-Yau type,
Mirror Symmetry, AMS/IP Advanced Studies in Math 38 (2006), 539-559.

\bibitem{VV}
Vasconcelos, WV and Villarreal, R,
On Gorenstein ideals of codimension four,
Proceedings of the American Mathematical Society
98 no. 2 (1986), 205 - 210.

\end{thebibliography}

\end{document}